\documentclass[a4paper,11pt]{amsart}
\usepackage{amssymb,amsfonts,amsmath}
\def\d{\delta}

\def\C{\mathbb{C}}
\def\c2{\mathbb{C}^2}
\def\R{\mathbb{R}}
\def\Q{\mathbb{Q}}

\def\N{\mathbb{N}}

\def\P{\mathbb{P}}

\def\1{\mathbf{1}}

\def\a{\alpha}
\def\b{\beta}
\def\e{\varepsilon}

\def\E{\mathcal{E}}

\def\f{\varphi}

\def\p{\psi}

\def\om{\omega}

 \newcommand{\MA}{\mathrm{MA}}

\newtheorem{lem}{Lemma}[section]
\newtheorem{prop}[lem]{Proposition}
\newtheorem{defi}[lem]{Definition}
\newtheorem{def/not}[lem]{Definition/Notations}

\newtheorem{thm}[lem]{Theorem}

\newtheorem{cor}[lem]{Corollary}
\newtheorem{rem}[lem]{Remark}

\newtheorem{exa}[lem]{Example}

\numberwithin{equation}{section}

\begin{document}

\title[Finite energy classes]
{Choquet-Monge-Amp\`ere classes }

\author{Vincent GUEDJ*, Sibel SAHIN** and Ahmed ZERIAHI*}

 \date{\today \\ *These authors are partially supported by the ANR project GRACK\\**This author is supported by TUBITAK 2219 Grant}

\address{Institut Universitaire de France  \& Institut Math\'ematiques de Toulouse, \\
Universit{\'e} Paul Sabatier\\ 31062 Toulouse cedex 09\\ France}

\email{vincent.guedj@math.univ-toulouse.fr}

\address{\"{O}zye\v{g}in University (Istanbul) \& Institut de Math\'ematiques de Toulouse,   \\ Universit\'e Paul Sabatier \\
118 route de Narbonne \\
F-31062 Toulouse cedex 09\\}

\email{sibel.sahin@ozyegin.edu.tr}

\address{Institut de Math\'ematiques de Toulouse,   \\ Universit\'e Paul Sabatier \\
118 route de Narbonne \\
F-31062 Toulouse cedex 09\\}

\email{ahmed.zeriahi@math.univ-toulouse.fr}

\begin{abstract}
We introduce and study Choquet-Monge-Amp\`ere classes on compact K\"ahler manifolds.
They consist of quasi-plurisubharmonic functions whose sublevel sets have
small enough asymptotic Monge-Amp\`ere capacity.
We compare them with  finite energy classes, which have recently played an important
role in  K\"ahler Geometry.
\end{abstract}

\maketitle

\section*{Introduction}

Let $(X,\omega)$ be a compact K\"ahler manifold of complex dimension $n \geq 1$.
Recall that a quasi-plurisubharmonic function (qpsh for short) on $X$ is
an upper semi-continuous function $\f:X \rightarrow \R \cup \{-\infty\}$
which is locally the sum of a plurisubharmonic and a smooth function.
We write $\f \in PSH(X,\omega)$ if $\f$ is qpsh and
$\omega_\f:=\omega+dd^c \f$ is a positive current.
Here $d=\partial+\overline{\partial}$ and $d^c =\frac{i}{2\pi}(\partial-\overline{\partial})$
are both real operators, so that $dd^c=\frac{i}{\pi} \partial\overline{\partial}$.

There are various ways to measure the singularities of such functions.
We can measure the asymptotic size of the sublevel sets
$(\f<-t)$ as $t \rightarrow +\infty$ through the Monge-Amp\`ere capacity,
which is defined by
$$
Cap_{\omega} (K) := \sup \left\{ \int_K MA(u) ; u \in PSH (X,\omega), - 1 \leq u \leq 0 \right\}.
$$
where $MA(u)=\omega_u^n/\int_X \omega^n$ is a well-defined probability measure \cite{BT82}.

We can also consider the non-pluripolar measure
$$
MA(\f):=\lim_{j \rightarrow +\infty} \1_{\{\f>-j\}} MA( \max(\f,-j)).
$$
Following \cite{GZ07} we say that $\f \in \E(X,\omega)$ if $MA(\f)$ is a probability measure and set
$$
\E^p(X,\omega):=\{ \f \in \E(X,\omega) \, | \, \f \in L^p(MA(\f)) \}.
$$

The finite energy classes $\E^p(X,\omega)$ have played an important role in recent applications of pluripotential theory to K\"ahler geometry (see for ex. \cite{EGZ08,EGZ09,BBGZ13,BEG13}).
While their local analogues can be well understood by measuring the size of the sublevel sets,
this is not the case in the compact setting (see \cite{BGZ08,BGZ09}).

In this article we introduce and initiate the study of the Choquet-Monge-Amp\`ere classes
$$
\mathcal  Ch^p (X,\omega) := \left\{ \f \in PSH (X,\omega)  \; | \;
 \int_0^{+ \infty} t^{p+n - 1} C_{\omega} (\{ \f \leq - t \}) d t<+\infty \right\}.
$$

We show in Theorem \ref{thm:caract} that an $\omega$-psh function $\f$ belongs to
$\mathcal  Ch^p (X,\omega)$ if and only if it has finite Choquet energy
$$
\rm{Ch}_p (\f) := \int_X (-\f)^p \left[ (-\f) \omega+\omega_\f \right]^n <+\infty.
$$

We establish in Corollary  \ref{cor:comparaison} that Choquet classes compare to finite energy classes as follows,
$$
\mathcal E^{p + n - 1} (X,\omega) \subset \mathcal Ch^p(X,\omega) \subset \mathcal E^p (X,\omega).
$$

These classes therefore coincide in dimension $n=1$, but the inclusions are strict in general when
$n \geq 2$: the first inclusion is sharp for functions with divisorial singularities (Proposition \ref{prop:div}),
while the second inclusion is sharp for functions with compact singularities (Proposition \ref{prop:compact}).

We briefly describe the range of the complex Monge-Amp\`ere operator acting on Choquet classes
in Proposition \ref{prop:range} and Proposition \ref{pro:MA2}. This description is not as complete as the
corresponding one for finite energy classe in \cite{GZ07};
Choquet classes are rather meant to become a useful intermediate tool in the analysis
of the complex Monge-Amp\`ere operator.

\section{Choquet classes}

\subsection{Choquet capacity}

\subsubsection{Generalized capacities}

Let $\Omega$ be a Hausdorff
locally compact topological space which we  assume is $\sigma$-compact.
We denote by $2^{\Omega}$ the set of all subsets of $\Omega$.
A set function $ c : 2^{\Omega} \longrightarrow \bar{\R}^+ := [0, + \infty]$ is
called a capacity on $\Omega$ if it satisfies the following four properties:

\smallskip

$(i)$ $c(\emptyset) = 0$;

\smallskip

$(ii)$ $c$ is monotone, i.e.
$
 A \subset B \subset \Omega \Longrightarrow 0 \leq c(A) \leq c(B);
$

$(iii)$  if $(A_n )_{n \in \N}$ is a non-decreasing sequence of subsets of $\Omega$, then
$$
c (\cup_n A_n) = lim_{n \to + \infty} c(A_n ) = \sup_n c(A_n );
$$

$(iv)$ if $(K_n )$ is a non-increasing sequence of compact subsets of
$\Omega$,
$$
c(\cap_n K_n )  = lim_n c(K_n) = \inf_n c(K_n).
$$

A capacity $c$ is said to be a Choquet capacity  if it is subadditive, i.e. if it satisfies the following extra condition

$(v$)  if $(A_n)_{n \in \N}$ is any sequence of subsets of $\Omega$, then
$$
c(\cup_n A_n) \leq \sum_n c (A_n),
$$

Capacities are usually first defined  for Borel subsets and then extended to all sets by building the
appropriate outer set function.

 \begin{exa} \label{exa:notcapacity}
 Let $\mathcal M$ be a family of Borel measures on $\Omega$. The set function defined on any Borel subsets
$A \subset \Omega$ by the formula
$$
c_{\mathcal M} (A) := sup \{ \mu (A); \mu \in \mathcal M\}
$$
is a precapacity on $\Omega$. It is called the upper envelope of  $\mathcal M$.
Observe that this precapacity need not be additive.
The precapacity $c_{\mathcal M}$ need not be outer regular either unless $\mathcal M$ is a  finite set.
However if $\mathcal M$ is a compact set for the weak$^*$-topology then $c_{\mathcal M}^*$ is a Choquet capacity.
\end{exa}

If  $c$ is a Choquet capacity on $\Omega$,  every Borel subset $B \subset \Omega$ satisfies
$$
c (B) = \sup \{ c (K) ; K \, \, \text{compact} \, \, K \subset B\}.
$$
  This is a special case of Choquet's capacitability theorem.

\subsubsection{Monge-Amp\`ere capacities}

Let $(X,\omega)$ be a compact K\"ahler manifold of dimension $n$.
We let $L^p (X)=L^p(X,\R,dV)$ denote the Lebesgue space of real valued measurable functions
which are $L^p$-integrable with respect to a fixed volume form $dV$.

We denote by $PSH (X,\omega)$ the set of $\omega$-plurisubharmonic functions: these are
functions $\f:X \rightarrow \R \cup \{ -\infty\}$ which are locally the sum of a plurisubharmonic
and a smooth function, and such that $\omega_\f:=\omega+dd^c \f \geq 0$ in the sense of currents.
 Recall that for all $p \geq 1$,
$$
 PSH (X,\omega) \subset L^p (X).
$$

The Monge-Amp\`ere capacity $C_{\omega}$ is defined for Borel sets $K \subset X$ by
$$
C_{\omega} (K) := \sup \left\{ \int_K MA(u) ; u \in PSH (X,\omega), - 1 \leq u \leq 0 \right\}.
$$
where $MA(u)=\omega_u^n/\int_X \omega^n$ is a well-defined probability measure \cite{BT82}.

It follows from the work of Bedford-Taylor \cite{BT82} that $C_{\omega}$ is a Choquet capacity.
We refer the reader to \cite{GZ05} for basics on $PSH(X,\omega)$ and $C_{\omega}$.

\subsection{The Choquet integral}

Let $C$ be a Choquet capacity on $X$, a compact topological space.

\begin{defi}
The Choquet's integral of a non negative Borel function $f : X \longrightarrow \R^+$ is
$$
\int_X f d C := \int_0^{+ \infty} C (\{ f \geq t\}) d t.
$$
\end{defi}

A change of variables shows that for any exponent $p \geq 1$,
$$
 \int_X f^p d C := p \int_0^{+ \infty} t^ {p - 1} C (\{ f \geq t\}) d t.
$$

Observe that if $K \subset X$ is a Borel set then
$$
 \int_X {\bf 1}_K d C = C (K).
$$

\begin{defi}
We set, for $p \geq 1$,
$$
 \mathcal L^p (X,C) := \{ f \in \mathcal B (X,\R) ;  \Vert f \Vert_{L^p (X,C)}  < + \infty \},
$$
where $\mathcal B (X,\R)$  is the space of real-valued Borel functions in $X$ and
$$
\Vert f \Vert _{L^p (X,C)} := \left( \int_X \vert f\vert^p d C \right)^{1/p}.
$$
\end{defi}

\begin{lem}
Let  $f, g \in \mathcal L^p (X,C)$ and $\lambda \in \R$.  Then

1. If $ 0 \leq f \leq g$ then  $\int_X f d C \leq \int_X g d C$.

\smallskip

2.
 $
 \Vert \lambda f \Vert_{L^p (X,C)} = \vert \lambda \vert  \Vert f \Vert_{L^p (X,C)}.
$

\smallskip

3.
$
 \Vert f + g \Vert_{L^p (X,C)} \leq 2 (\Vert f \Vert_{L^p (X,C)} + \Vert g \Vert_{L^p (X,C)}),
$
\end{lem}

In particular $\mathcal L^p (X,C)$ is a vector space.
The above quasi-triangle inequality  defines a uniform structure,
which is furthermore metrizable for general reasons \cite{BOUR}, i.e.
 we can equip $\mathcal L^p (X,C)$ with an invariant metric $\rho$
such that a sequence $(f_j)$ converges to $f$ for the metric $\rho$ if and only if
$\lim_{j \to + \infty} \Vert  f_j - f \Vert_{L^p (X,C)} = 0$.

\begin{proof}
The first two items are obvious. The third one follows from the subaddivity of the capacity and the inclusion
$$
\{ f+g \geq t \} \subset \{ f \geq \frac{t}{2} \} \cup \{ g \geq \frac{t}{2} \}.
$$
\end{proof}

\begin{lem} \label{lem:continuity}
Let $(f_j)$ be a  sequence of non-negative Borel functions on $X$.

1. If $(f_j)$ is non-decreasing  and $f := \sup_j f_j$, then
$$
 \int_X f d C = \lim_{j \to + \infty} \int_X f_j d C = \sup_j \int_X f_j d C.
$$

2. If $(f_j)$ is a decreasing sequence of positive upper semi-continuous functions
 and $f := \inf_j f_j$, then
$$
 \int_X f d C = \lim_{j \to + \infty} \int_X f_j d C = \inf_j \int_X f_j d C.
$$
\end{lem}

This lemma follows from continuity properties of the Choquet capacity;
the proof  is left to the reader.

\subsection{Choquet-Monge-Amp\`ere classes}

Let $(X,\omega)$ be a compact K\"ahler manifold of dimension $n$.

\begin{defi}
The Choquet-Monge-Amp\`ere class is
$$
\mathcal  Ch^p (X,\omega) := PSH (X,\omega) \cap \mathcal L^{p+n} (X,C_{\omega}).
$$

\end{defi}

Observe that when $\f \in \mathcal Ch^p (X,\omega)$ and $\f \leq 0$  then

$$
 \int_X (- \f)^{p+n} d C_{\omega} = (p+n) \int_0^{+ \infty} t^{p+n - 1} C_{\omega} (\{ \f \leq - t \}) d t,
$$
and
$$
 C_{\omega} (\{ \f \leq - t \})  \leq t^{- p-n} \int_X (- \f)^{p+n} d C_{\omega}.
$$

\begin{prop}  \label{pro:convex}
The class  $\mathcal Ch^p (X,\omega)$ is convex.

If $(\f_j) \in \mathcal Ch^p (X,\omega)^{\N}$
converges in $L^ 1 (X)$ to $\f \in PSH (X,\omega)$  and satisfies
$
  \sup_j \int_X (- \f_j)^{p+n} \, d C_{\omega} < + \infty,
 $
then $ \f \in \mathcal Ch^p (X,\omega)$ and
$$
\int_X (- \f)^{p+n} d C_{\omega} \leq \liminf_{j \to + \infty}  \int_X (- \f_j)^{p+n} d C_{\omega}.
$$
\end{prop}

\begin{proof}
Set
$$
 \tilde \f_j := \left(\sup_{\ell \geq j} \f_{\ell} \right)^*.
$$
Then $(\tilde \f_j)$ is a non-increasing sequence of $PSH(X,\omega)$ which converges to $\f$
pointwise.
Since $\f_j \leq \tilde \f_j \leq 0$ for all $j$, we infer that  $\tilde \f_j \in \mathcal Ch^p (X,\omega)$ and
$$
 \int_X (- \tilde \f_j)^{p+n} d C_{\omega}  \leq  \int_X (- \f_j)^{p+n} d C_{\omega}
\leq M:=  \sup_j \int_X (- \f_j)^{p+n} \, d C_{\omega}.
$$
By Lemma \ref{lem:continuity} we conclude that
\begin{eqnarray*}
\int_X (- \f)^{p+n} d C_{\omega} &= &\lim_j \int_X (- \tilde \f_j)^{p+n} d C_{\omega}  \\
&\leq&  \liminf_j \int_X (-\f_j)^{p+n} d C_{\omega} \leq M.
\end{eqnarray*}
Hence $\f \in \mathcal Ch^p (X,\omega)$ and the required inequality follows.
\end{proof}

\section{Energy estimates}

We now compare the Choquet-Monge-Amp\`ere classes with the finite energy classes
${\mathcal E}^q(X,\omega)$ introduced in \cite{GZ07}.

\subsection{Finite energy classes}

Given   $\f \in PSH(X,\omega)$, we consider its canonical approximants
$$
\f_j:=\max(\f, -j) \in PSH(X,\omega) \cap L^{\infty}(X).
$$
It follows from the Bedford-Taylor theory that the measures $MA(\f_j)$ are well defined probability measures.
 Since the $\f_j$'s are decreasing, it is natural to expect that these measures converge (in the weak sense).
The following strong monotonicity property holds:

\begin{lem}
The sequence
$
\mu_j:={\bf 1}_{\{ \f>-j\}} MA(\f_j)
$
is an increasing sequence of Borel measures.
\end{lem}

 The proof is an elementary consequence of the maximum principle (see \cite[p.445]{GZ07}).
Since the $\mu_j$'s all have total mass bounded from above by $1$ (the total mass of the measure $MA(\f_j)$),  we can consider
$$
\mu_{\f}:=\lim_{j \rightarrow +\infty} \mu_j,
$$
which is a positive Borel measure on $X$, with total mass $\leq 1$.

\begin{defi}
We set
$$
{\mathcal E}(X,\omega):=\left\{ \f \in PSH(X,\omega) \; | \; \mu_{\f}(X)=1 \right\}.
$$
For $\f \in {\mathcal E}(X,\omega)$, we set
$
MA(\f):=\mu_{\f}.
$
\end{defi}

The notation is justified by the following important fact:
the complex Monge-Amp\`ere operator $\f \mapsto MA(\f)$ is well defined on the class
${\mathcal E}(X,\omega)$, i.e. for every decreasing sequence of bounded (in particular
smooth) $\omega$-psh functions $\f_j$, the probability measures
$MA(\f_j)$ weakly converge towards $\mu_\f$, if $\f \in {\mathcal E}(X,\omega)$.

\smallskip

Every bounded $\omega$-psh function clearly belongs to ${\mathcal E}(X,\omega)$
since  in this case $\{\f >-j\}=X$ for $j$ large enough, hence
$$
\mu_{\f} \equiv \mu_j=MA(\f_j)=MA(\f).
$$
The class ${\mathcal E}(X,\omega)$ also contains many $\omega$-psh functions which are unbounded.
When $X$ is a compact Riemann surface ($n=\dim_{\C} X=1$),
the set ${\mathcal E}(X,\om)$ is the set of $\om$-sh functions
whose Laplacian does not charge polar sets.

\begin{rem}
If $\f \in PSH(X,\omega)$ is normalized so that $\f \leq -1$, then $-(-\f)^\e $ belongs to
${\mathcal E}(X,\om)$ whenever $0 \leq \e <1$.
The functions which belong to the class ${\mathcal E}(X,\om)$,
although usually unbounded,
have relatively mild singularities. In particular they have zero Lelong numbers.
\end{rem}

It is shown in \cite{GZ07} that  the comparison principle holds in ${\mathcal E}(X,\omega)$:

\begin{prop} \label{pro:CP}
Fix $u,v \in \E(X,\omega)$. Then
$$
\int_{\{v<u\}}  MA(u) \leq \int_{\{v<u\}}  MA(v).
$$
\end{prop}

The class ${\mathcal E}(X,\omega)$ is  the largest class for which the complex Monge-Amp\`ere is well defined and the maximum principle holds.

\begin{defi}
We let
${\mathcal E}^p(X,\om)$ denote the set of
$\om$-psh
functions with
finite $p$-energy, i.e.
$$
{\mathcal E}^p(X,\om):=\left\{ \f \in {\mathcal E}(X,\om) \, / \,  (|\f|)^p \in L^1( MA(\f)) \right\}.
$$
\end{defi}

Here follows a few important properties of these classes (see \cite{GZ07}):
\begin{itemize}
\item when $p \geq 1$, any $\f \in {\mathcal E}^p(X,\om)$ is such that
$\nabla_\omega \f \in L^2(\om^n)$;
\item $\f \in \E^p(X,\omega)$ if and only if for any (resp. one)  sequence of bounded
$\omega$-functions decreasing to $\f$,
$\sup_j \int_X (-\f_j)^p MA(\f_j) <+\infty$;
\item the class $\E^p(X,\omega)$ is  convex.
\end{itemize}

\subsection{Choquet energy}

For $\f \in PSH^- (X,\omega)$ and $ p \geq 1$, we set
$$
\rm{Ch}_p (\f) := \sum_{j = 0}^{n}  C_n^j
 \int_X (- \f)^ {p + j} \omega_{\f}^ {n - j} \wedge \omega^j
=\int_X (-\f)^p \left[ (-\f) \omega+\omega_\f \right]^n.
$$
Here and in the sequel we use the french notation $C_n^j:=\left( \begin{array}{c} n \\ j \end{array} \right)$.

We recall the  following  useful result:

\begin{lem} \label{lem:Cap}
Fix  $\f,\p \in \mathcal E (X,\om)$.
Then for all $t  < 0$ and $0 \leq \d \leq 1$,
$$
\d^n C_{\om}(\{\f-\p<-t-\d\}) \leq \int_{\{\f-\p<-t- \d \p\}} MA (\f).
$$

In particular
$$
\d^n C_{\om}(\{\f <-t-\d\}) \leq \int_{\{\f <-t\}} MA (\f).
$$
\end{lem}

\begin{proof}
If $u$ is a $\omega$-psh function such that $0 \le u \le 1$, then
$$
\{\f<-t-\d\}\subset\{\f<\d u-t-\d\}\subset\{\f<-t\}.
$$

Since $\d^n\MA(u)\le\MA(\d u)$ and $\f \in {\mathcal E}(X,\omega)$
it follows from the comparison principle that
\begin{eqnarray*}
 \lefteqn{\hskip-2cm  \d^n\int_{\{\f<-t-\d\}}\MA(u)  \le \int_{\{\f<\d u-t-\d\}}\MA(\d u)}  \\
&& \le\int_{\{\f<\d u-t-\d\}}\MA(\f)
 \le\int_{\{\f<-t\}}\MA(\f).
\end{eqnarray*}

This proves the last inequality. The first one is a refinement of the first, we refer the reader to
\cite{EGZ09} for a proof.
\end{proof}

\begin{thm} \label{thm:caract}
For all $p \geq 1$ and $0 \geq \f \in PSH (X,\om) \cap L^{\infty}(X)$,
\begin{equation} \label{eq:compIneq1}
 \int_X (-\f)^{n + p} d C_{\omega} \leq  2^{n + p} \rm{Ch}_p (\f),
\end{equation}
 and
 \begin{equation} \label{eq:compIneq2}
 \rm{Ch}_p (\f)  \leq   V_{\omega} (X) +  (n + 1) 2^n \int_X (-\f)^{n + p} d C_{\omega}.
\end{equation}
In particular
$$
 \mathcal Ch^p (X,\omega) = \{ \f \in PSH(X,\omega) ; \rm{Ch}_p (\f) < + \infty\},
$$
and the inequalities (\ref{eq:compIneq1}) and (\ref{eq:compIneq2}) hold for all
$\f \in \mathcal Ch^p (X,\omega)$.
\end{thm}

\begin{proof}
By Lemma \ref{lem:continuity} and the continuity properties for the Monge-Amp\`ere operators,
it suffices to prove the estimates  (\ref{eq:compIneq1}) and  (\ref{eq:compIneq2}) when
$0 \geq \f \in  PSH (X,\omega) \cap L^{\infty} (X)$.
Now
\begin{eqnarray*}
\int_X (-\f)^{n + p} d C_{\omega} = (n + p) \int_0^{+ \infty} t^{n + p - 1} C_{\omega} (\{ \f \leq - t\}) d t.
\end{eqnarray*}

Fix  $t \geq 1$ and $u \in PSH (X,\omega)$ such that $- 1 \leq u \leq 0$.
Observe that $\f \slash t \in  PSH^- (X,\omega) \cap L^{\infty}(X)$ and
$$
\{\f < - 2 t\} \subset \{\f \slash t < u - 1\} \subset \{ \f < - t\}.
$$

Set $\psi_t := \f \slash t$. This is a bounded $\omega$-psh function in $X$ such that $\omega + dd^c \psi_t \leq t^{- 1} \omega_\f + \omega$.
The comparison principle (Proposition \ref{pro:CP}) yields
$$
\int_{ \{\f < - 2 t\}} \omega_u^n \leq \int_{ \{\psi_t < u - 1\}} \omega_u^n \leq
\int_{ \{\f < - t\}} (t^{- 1} \omega_\f + \omega)^n.
$$
Since
$
(t^{- 1} \omega_\f + \omega)^n = \sum_{j = 0}^n C_n^j t^{-n+j} \omega_\f^{n - j} \wedge \omega^j,
$
we infer, for all $t \geq 1$,
$$
t^n C_{\omega} (\{\f < - 2 t\}) \leq \sum_{j = 0}^n
C_n^j t^{j} \int_{\{\f < - t\}} \omega_{\f}^{n - j} \wedge \omega^j.
$$

 It follows on the other hand  from Lemma \ref{lem:Cap} that for $0 < t \leq 1$,
 $$
 t^n C_{\omega} (\{\f < - 2 t\}) \leq  \int_{\{\f < - t\}} \omega_\f^n.
 $$
Thus for all $t > 0$
$$
  t^{n + p -1} C_{\omega} (\{\f < - 2 t\}) \leq 2 \sum_{j = 0}^n C_n^j (p+ j)  t^{p +j - 1} \int_{\{\f < - t\}} \omega_{\f}^{n - j} \wedge \omega^j,
$$
hence
$$
\int_X (- \f)^{n + p} d C_{\omega} \leq (n + 1) 2^{n + p + 1} \rm{Ch}_p (\f).
$$

Conversely fix $\f \in PSH (X,\omega) \cap L^{\infty} (X)$. Then for $j = 0, \cdots, n$
{\small
 $$
  \int_X (- \f)^{p + j} \omega_\f^{n - j} \wedge \omega^j =  V_{\omega} (X)
   +  (p + j) \int_1^{+ \infty} t^{p + j - 1} \omega_\f^{n - j} \wedge \omega^j (\{\f \leq - t\}).
$$
}

Observe that if we set $\f_t := \sup \{\f , - t \}$, then
$$
\int_{\{\f \leq - t\}} \omega_\f^{n - j} \wedge \omega^j = \int_{\{\f \leq - t\}} \omega_{\f_t}^{n - j} \wedge \omega^j.
$$
Since for $t \geq 1$,  $t^{- 1} \omega_{\f_t} \leq \omega_{\psi_t}$, where
 $\psi_t := \sup \{\f \slash t, - 1\}$ , we infer
$$
\int_{\{\f \leq - t\}} \omega_\f^{n - j} \wedge \omega^j \leq t^{n - j} \int_{\{\f \leq - t\}} \omega_{\psi_t}^{n - j} \wedge \omega^j.
$$
Now
$$
 C_n^j \omega_{\psi_t}^{n - j} \wedge \omega^j \leq  (\omega_{\psi_t} + \omega)^n = 2^n (\omega + dd^c (\psi_t \slash 2))^n
$$
and  $- 1 \leq \psi_t \leq 0$, therefore
$$
\sum_{j = 0}^n  C_n^j \int_0^{+ \infty} t^{p + j} \omega_{\f}^{n - j} \wedge \omega^j \leq 2^n V_{\omega} (X) + n 2^n \int_X (-\f)^{n + p} d C_{\omega},
$$
hence
$$
\sum_{j = 0}^n  C_n^j \int_X (-\f)^{p + j} \omega_{\f}^{n - j} \wedge \omega^j \leq 2^n V_{\omega} (X) + n 2^n \int_X (-\f)^{n + p} d C_{\omega}.
$$
\end{proof}

\begin{cor}    \label{cor:comparaison}
$$
\mathcal E^{p + n - 1} (X,\omega) \subset \mathcal Ch^p(X,\omega) \subset \mathcal E^p (X,\omega).
$$
\end{cor}

\begin{proof}
The second inclusion follows from the fact that
$$
\int_X (- \f)^p \omega_\f^n \leq \rm{Ch}_p (\f).
$$

To prove the first inclusion we can assume that $\f \leq - 1$.  Observe that when $\f \in \mathcal E^{p + n - 1} (X,\omega)$ so does $\f \slash 2$ and for $j = 1, \cdots, n-1$
$$
 \int_X (- \f)^{p + j} \omega_\f^{n - j} \wedge \omega^j \leq 2^n \int_X (- \f)^{p + j} \omega_{\f\slash 2}^{n}
$$
and for $j = 0$, we always have $ \int_X (- \f)^{p} \omega_\f^{n} < + \infty$.
\end{proof}

\section{Range of the Monge-Amp\`ere operator}

In this section $X$ is a compact K\"ahler manifold equipped with  a semi-positive form
$\omega$ such that $\int_X \omega^n=1$, where $n=\dim_{\C} X$.

\subsection{The Monge-Amp\`ere operator on  $\mathcal Ch^p (X,\omega)$}

\begin{lem}
Fix  $0 \geq \f, \psi \in \mathcal Ch^p (X,\omega)$ and $0 \leq j \leq n$. Then
$$
\int_X (-\f)^{p+j} \omega_{\psi}^{n - j} \wedge \omega^j \leq 2^{p+j}  \int_X (-\f)^{p+j} \omega_{\f}^{n - j} \wedge \omega^j +  2^{p+j} \int_X (-\psi)^{p+j} \omega_{\psi}^{n - j} \wedge \omega^j.
$$
\end{lem}

\begin{proof}
Set $\chi(t)=-(-t)^{p+j}$.
The proof is slightly different if $j=0$ and $0 < p<1$ or if $p+j \geq 1$
($\chi$ is convex or concave). We only treat the second case and leave the modifications to the reader.
Observe that
$0 \leq \chi'(2t) =M  \chi'(t)$, with $M=2^{p+j-1}$, hence
\begin{eqnarray*}
\int_X (-\chi) \circ \f \, \om_{\p}^{n-j} \wedge \omega^j&=&
\int_{-\infty}^0 \chi'(t) \om_{\p}^{n-j} \wedge \omega^j (\f <t) dt \\
&\leq&  2M\int_{-\infty}^0 \chi'(t) \om_{\p}^{n-j} \wedge \omega^j (\f <2t) dt.
\end{eqnarray*}
Now $(\f <2t) \subset (\f<\p+t) \cup (\p<t)$, hence
$0 \leq \chi'(2t) =M  \chi'(t)$, with $M=2^{p+j-1}$, hence
\begin{eqnarray*}
\int_X (-\chi) \circ \f \, \om_{\p}^{n-j} \wedge \omega^j&\leq&
2M \int_{-\infty}^0 \chi'(t) \om_{\p}^{n-j} \wedge \omega^j (\f <\p+t) dt \\
&+&  2M \int_X (-\psi)^{p+j} \omega_{\psi}^{n - j} \wedge \omega^j.
\end{eqnarray*}
The comparison principle yields
$\om_{\p}^{n-j} \wedge \omega^j (\f<\p+t) \leq  \om_{\f}^{n-j} \wedge \omega^j (\f<\p+t)$.
The desired inequality follows by observing that
$(\f<\p+t) \subset (\f<t)$.
\end{proof}

\begin{lem}
Let $\mu$ be a probability measure. Then
$
\mathcal Ch^p (X,\omega) \subset L^q(\mu)
$
if and only if there exists $ C_\mu>0$ such that
$\forall \f \in \mathcal Ch^p (X,\omega) \text{ with } \sup_X \f=-1, \; $
$$
\int_X (-\f)^q \, d\mu \leq C_\mu \left[ \rm{Ch}_p(\f) \right]^{\frac{q}{p+n}}.
$$
\end{lem}

\begin{proof}
One implication is obvious. Assume that $\mathcal Ch^p (X,\omega) \subset L^q(\mu)$, we want to establish
the quantitative integrability property.
Assume on the contrary that there exists a sequence $\f_j \in \mathcal Ch^p (X,\omega) $ with
$\sup_X \f_j=-1$ and
$$
\int_X (-\f_j)^q \, d\mu \geq 4^{jq} \rm{Ch}_p(\f_j)^{\frac{q}{p+n}}.
$$

Assume first that $M_j:=\rm{Ch}_p(\f_j)$ is uniformly bounded.
Note that $M_j \geq 1$ since $\f_j \leq -1$.
It follows from Proposition \ref{pro:convex}
that $\f=\sum_{j \geq 1} 2^{-j} \f_j$ belongs to $\mathcal Ch^p (X,\omega)$. Now for all $k \geq 1$,
$$
\int_X (-\f)^q \, d\mu \geq 2^{-kq} \int_X (-\f_k)^q \, d\mu
\geq 2^{kq},
$$
hence $\int_X (-\f)^q \, d\mu=+\infty$, a contradiction.

Extracting and relabelling we can thus assume  $M_j:=\rm{Ch}_p(\f_j) \rightarrow +\infty$.
Set $\p_j=\e_j \f_j$ with $\e_j=M_j^{-\frac{1}{n+p}}$ and $\p=\sum_{j \geq 1} 2^{-j} \p_j$.
We note again that for all $k \geq 1$,
$$
\int_X (-\p)^q \, d\mu \geq 2^{-kq} \int_X (-\p_k)^q \, d\mu
\geq 2^{kq} \e_k^q M_k^{\frac{q}{p+n}}=2^{kq},
$$
hence $\p \notin L^q(\mu)$. We now show that $\p \in \mathcal Ch^p (X,\omega)$ to get a  contradiction.
It suffices to show that $\rm{Ch}_p(\p_j)$ is uniformly bounded from above. Observe that $\omega_{\p_j}=\e_j \omega_{\f_j}+(1-\e_j) \omega \leq \e_j \omega_{\f_j}+\omega$. We need to control each term
$$
\e_j^{p+n-k}   \int_X (-\f_j)^{p+\ell} \omega_{\f_j}^{n-\ell-k} \wedge \omega^{\ell+k},
$$
where $0 \leq \ell \leq n$ and $0 \leq k \leq n- \ell$. H\"older inequality yields
$$
    \int_X (-\f_j)^{p+\ell} \omega_{\f_j}^{n-\ell-k} \wedge \omega^{\ell+k}
\leq \left(  \int_X (-\f_j)^{p+\ell+k} \omega_{\f_j}^{n-\ell-k} \wedge \omega^{\ell+k}
\right)^{\frac{p+\ell}{p+\ell+k}},
$$
therefore
$$
\rm{Ch}_p(\p_j) \leq C \max_{\ell,k} \left( \e_j^{p+n-k} M_j^{\frac{p+\ell}{p+\ell+k}} \right)
=C \max_{\ell,k} \left( M_j^{-\frac{k(n-\ell-k)}{(p+\ell+k)(n+p)}} \right)
\leq C',
$$
since $\e_j=M_j^{-\frac{1}{n+p}}$.
\end{proof}

The range of the complex Monge-Amp\`ere operator acting on finite energy classes has been
characterized in \cite{GZ07}. The situation is more subtle for Choquet-Monge-Amp\`ere classes.

We now connect the way a non pluripolar measure is dominated by the Monge-Amp\`ere capacity to
integrability properties with respect to Choquet-Monge-Amp\`ere classes:

\begin{prop} \label{prop:range}
Let $\mu$ be  a probability measure on $X$.
If $\mu \leq A \, C_\om^\a$ with  $0<A$ and $ q/(p+n)<\a<1$, then
$$
\mathcal Ch^p (X,\omega) \subset L^q(\mu).
$$

\end{prop}

\begin{proof}
 Let $\f \in {\mathcal Ch}^p(X,\om)$
with $\sup_X \f=-1$. It follows from H\"older inequality that
{\small
\begin{eqnarray*}
\lefteqn{ \! \! \! \! \! \! \! \! \!
0  \leq  \int_X (-\f)^q d\mu =1+ q \int_1^{+\infty} t^{q-1} \mu(\f<-t) dt} \\
&\leq& 1+qA \int_1^{+\infty}t^{q-1} \left[ Cap_{\om}(\f<-t) \right]^{\a} dt \\
&\leq& 1+qA \left[ \int_1^{+\infty} t^{\frac{q-\a(p+n)}{1-\a}-1} dt \right]^{1-\a}
\cdot \left[ \int_1^{+\infty} t^{n+p-1} Cap_{\om}(\f<-t) dt \right]^{\a}.
\end{eqnarray*}
}

The first integral in the last line converges when $q/(p+n)<\a$ since $q-\a(p+n)<0$.
The last one  is bounded from above by definition.
Therefore ${\mathcal Ch}^p(X,\om) \subset L^q(\mu)$.
\end{proof}

We now investigate conditions under which the converse of this result holds.
We start by considering the problem for the finite energy classes $\mathcal{E}^{p}(X,\omega)$ :

\begin{prop}\label{prop:measurecapacity}
If $\mathcal{E}^{^p}(X,\omega)\subset L^{p}(\mu)$ for $p>1$, then there exists an $A>0$
such that $\mu\leq A C_{\omega}^{\alpha}$ where $\alpha=(1-{1}/{p})^{n}$.
\end{prop}

\begin{proof}
Suppose that $\mathcal{E}^{^p}(X,\omega)\subset L^{p}(\mu)$ then by \cite{GZ07} $\mu=\omega^{n}_{\psi}$ for some $\psi\in \mathcal{E}^{^p}(X,\omega)$ such that $\sup_{X}\psi=-1$. Let $\varphi\in PSH(X,\omega)$ with $-1\leq \f \leq0$ then
\begin{eqnarray*}
\int_{X}(-\varphi)^{p}\omega_{\psi}^{n}&=&\int_{X}(-\varphi)^{p}\omega_{\psi}\wedge\omega_{\psi}^{n-1} \\
&=&\int_{X}(-\psi)(-dd^{c}(-\varphi)^{p})\wedge\omega_{\psi}^{n-1}+\int_{X}(-\varphi)^{p}\omega\wedge\omega_{\psi}^{n-1}.
\end{eqnarray*}
Now
$$
-dd^{c}(-\varphi)^{p}=-p(p-1)(-\varphi)^{p-2}d\varphi\wedge d^{c}\varphi+p(-\varphi)^{p-1}dd^{c}\varphi
\leq p(-\f)^{p-1} dd^c \f
$$
and $(-\f)^p \leq (-\f)^{p-1}$ since $0 \leq -\f \leq 1$, hence
\begin{eqnarray*}
\int_{X}(-\varphi)^{p}\omega_{\psi}^{n}
 &\leq&  p\int_{X}(-\psi)(-\varphi)^{p-1}dd^{c}\varphi\wedge\omega_{\psi}^{n-1}+\int_{X}(-\varphi)^{p-1}\omega\wedge\omega_{\psi}^{n-1} \\
&\leq&  p\int_{X}(-\psi)(-\varphi)^{p-1}dd^{c}\varphi\wedge\omega_{\psi}^{n-1}+\int_{X}(-\psi)(-\varphi)^{p-1}\omega\wedge\omega_{\psi}^{n-1} \\
 &=& p \int_{X}(-\psi)(-\varphi)^{p-1}\omega_{\varphi}\wedge\omega_{\psi}^{n-1}.
\end{eqnarray*}
H\"older inequality thus yields
\begin{eqnarray*}
\int_{X}(-\varphi)^{p}\omega_{\psi}^{n}
&\leq & p\left(\int_{X}(-\psi)^{p}\omega_{\varphi}\wedge\omega_{\psi}^{n-1}\right)^{\frac{1}{p}}\left(\int_{X}(-\varphi)^{p}\omega_{\varphi}\wedge\omega_{\psi}^{n-1}\right)^{1-\frac{1}{p}} \\
 &\leq & p\left(\int_{X}(-\psi)^{p}\omega_{\psi}^{n}\right)^{\frac{1}{p}}\left(\int_{X}(-\varphi)^{p}\omega_{\varphi}\wedge\omega_{\psi}^{n-1}\right)^{1-\frac{1}{p}}.
\end{eqnarray*}

Repeating the same argument $n$ times we end up with
\begin{equation*}
\int_{X}(-\varphi)^{p}\omega_{\psi}^{n}\leq A \left(\int_{X}(-\varphi)^{p}\omega_{\varphi}^{n}\right)^{(1-{1}/{p})^{n}}
\end{equation*}

Fix $E \subset  X$ a compact set. The conclusion follows by applying this inequality to the extremal function $\f=h^{*}_{\omega,E}$, observing that
$$
\mu(E) \leq \int_{X}(-h^{*}_{\omega,E})^{p}\omega_{\psi}^{n},
$$
while
$
C_{\omega}(E)=\int_{X}(-h^{*}_{\omega,E})^{p}\omega_{h^{*}_{\omega,E}}^{n},
$
as shown in \cite{GZ05}.
\end{proof}

\begin{lem}
 Let $\mu$ be a probability measure.
Then $\mathcal{E}^{p}(X,\omega)\subset L^{q}(\mu)$ if and only if  there exists a constant $C>0$ such that for all $\psi\in PSH(X,\omega)\cap L^{\infty}(X)$ with $\sup_{X}\psi=-1$
\begin{equation}\label{eq:staroflemma}
0\leq\int_{X}(-\psi)^{q}d\mu\leq C \left(\int_{X}(-\psi)^{p}\omega_{\psi}^{n}\right)^{\frac{q}{p+1}}
\end{equation}
\end{lem}

\begin{proof}
One implication is clear so suppose that $\mathcal{E}^{p}(X,\omega)\subset L^{q}(\mu)$ and assume for a contradiction that there exists $\psi_{j}\in PSH(X,\omega)\cap L^{\infty}(X)$ with $\sup_{X}\psi_{j}=-1$ such that
\begin{equation*}
\int_{X}(-\psi_{j})^{q}d\mu\geq4^{jq}M_{j}^{\frac{q}{p+1}}
\end{equation*}where $M_{j}=\int_{X}(-\psi_{j})^{p}\omega_{\psi_{j}}^{n}$.

If $M_{j}$ is uniformly bounded then  $\psi=\sum_{j\geq1} {2^{-j}} \psi_{j}$ belongs to $\mathcal{E}^{p}(X,\omega)$. Now
\begin{equation*}
\int_{X}(-\psi)^{q}d\mu\geq\int_{X}\frac{(-\psi_{j})^{q}}{2^{jq}}d\mu\geq2^{jq}M_{j}^{\frac{q}{p+1}}\geq2^{jq}
\end{equation*}
since $\psi_{j}\leq-1$, $M_{j}\geq1$. So $\int_{X}(-\psi)^{q}d\mu\rightarrow\infty$, a contradiction.

We obtain the same contradiction if $\{M_{j}\}$ admits a bounded subsequence so we can assume $M_{j}\rightarrow\infty$ and $M_{j}\geq1$. Set $\varphi_{j}=\varepsilon_{j}\psi_{j}$ where $\varepsilon_{j}=M_{j}^{-\frac{1}{1+p}}$ and $\psi=\sum_{j\geq1}  {2^{-j}} \varphi_{j}$ then,
\begin{equation*}
\int_{X}(-\psi)^{q}d\mu\geq\int_{X}\frac{(-\varphi_{j})^{q}}{2^{jq}}d\mu=2^{-jq}\varepsilon_{j}^{q}\int_{X}(-\psi_{j})^{q}d\mu\geq2^{jq}\rightarrow\infty
\end{equation*}
so $\psi\notin L^{q}(\mu)$.

We now check that $\varphi_{j}\in \mathcal{E}^{p}(X,\omega)$ to derive a contradiction.
Since $\omega_{\varphi_{j}}\leq \varepsilon_{j}\omega_{\psi_{j}}+\omega$, we get
\begin{eqnarray*}
\int_{X}(-\varphi_{j})^{p}\omega_{\varphi_{j}}^{n}
&=& \varepsilon_{j}^{p}\int_{X}(-\psi_{j})^{p}\omega_{\varphi_{j}}^{n} \\
&\leq & \varepsilon_{j}^{p}\left(\int_{X}(-\psi_{j})^{p}\omega^{n}+2^{n}\varepsilon_{j}\int_{X}(-\psi_{j})^{p}\omega_{\psi_{j}}^{n}\right)=O(1),
\end{eqnarray*}
because
\begin{equation*}
\int_{X}(-\psi_{j})^{p}\omega_{\psi_{j}}^{n}=\int_{X}(-\psi_{j})^{p}\omega\wedge\omega_{\psi_{j}}^{n-1}+\int_{X}p(-\psi_{j})^{p-1}d\psi_{j}\wedge\ d^{c}\psi_{j}\wedge\omega_{\psi_{j}}^{n-1}
\end{equation*}
\begin{equation*}
\geq\int_{X}(-\psi_{j})^{p}\omega\wedge\omega_{\psi_{j}}^{n-1}\geq...\int_{X}(-\psi_{j})^{p}\omega^{k}\wedge\omega_{\psi_{j}}^{n-k}
\end{equation*}
for all $1\leq k\leq n-1$ and $\int_{X}(-\psi_{j})^{p}\omega^{n}$ is bounded since $\psi_{j}\in PSH(X,\omega)\cap L^{\infty}(X)$ and $\sup_{X}\psi_{j}=-1$.
\end{proof}

We are now ready to give necessary conditions for a non-pluripolar measure to be dominated by the Monge-Amp\`ere capacity,
in terms of its integrability condition properties with respect to Choquet-Monge-Amp\`ere classes:

\begin{prop} \label{pro:MA2}
Let $\mu$ be a non-pluripolar probability measure such that $\mu=MA(\psi)$ where $\psi\in Ch^p (X,\omega)$, then $\mu\leq (Cap_{\omega})^{\frac{p}{p+n}}$
\end{prop}

\begin{proof}
From \cite{GZ07} we already know that $\mu=\omega_{\psi}^{n}$ for some function $\psi\in \mathcal{E}(X,\omega)$ such that $\sup_{X}\psi=-1$. Now suppose also that $\psi\in Ch^p (X,\omega)$. For $\f\in PSH(X,\omega)$ with $-1\leq\f\leq0$ H\"{o}lder inequality and integration by parts yields
$$
\int_{X}(-\f)^{p+n}\omega_{\psi}^{n}\leq(p+n)\int_{X}(-\f)^{p+n-1}(-\psi)\omega_{\f}\wedge\omega_{\psi}^{n-1}
$$
$$
\leq(p+n)\left(\int_{X}(-\psi)^{p+1}\omega_{\f}\wedge\omega_{\psi}^{n-1}\right)^{\frac{1}{p+1}}\left(\int_{X}(-\f)^{\frac{(p+n)(p+n-1)}{p}}\omega_{\f}\wedge\omega_{\psi}^{n-1}\right)^{\frac{p}{p+1}}
$$
To handle the second term observe that 
$$
\int_{X}(-\f)^{\frac{(p+n)(p+n-1)}{p}}\omega_{\f}\wedge\omega_{\psi}^{n-1}
$$
$$
\leq c_{p,n}\left(\int_{X}(-\psi)^{p+2}\omega_{\f}^{2}\wedge\omega_{\psi}^{n-2}\right)^{\frac{1}{p+2}}\left(\int_{X}(-\f)^{\frac{(p+2)(p^{2}+(p+1)(n-1))}{p(p+1)}}\omega_{\f}^{2}\wedge\omega_{\psi}^{n-2}\right)^{\frac{p+1}{p+2}}
$$
As it can be observed, at each step the power of $(-\f)$ is obtained by reducing the previous power by $1$ first and then multiplying by $\frac{p+m}{p+m-1}$ where $m$ is the number of the corresponding step. Hence the power of $(-\f)$ at the $m$'th step, $\sigma_{m}$, is given by induction by
$$
\sigma_{m+1}=\frac{p+m}{p+m-1}(\sigma_{m}-1)
$$
Therefore we have to justify that $\sigma_{m}$ is bigger than $p+n-m$ and we can continue the procedure $n$-times. We will show this by induction:\\
For $m=1$, $\sigma_{1}=\frac{p+1}{p}(p+n-1)>p+n-1$ since $\frac{p+1}{p}>1$ and assume that $\sigma_{m}>p+n-m$ then
$$
\sigma_{m+1}=\frac{p+m}{p+m-1}(\sigma_{m}-1)>\frac{p+m}{p+m-1}(p+n-m-1)>p+n-(m+1) 
$$ since $\frac{p+m}{p+m-1}>1$. \\
Now at the $n$'th step we have,
$$
\int_{X}(-\f)^{p+n}\omega_{\psi}^{n}\leq c_{p,n}\left(\prod_{i=1}^{n}\left(\int_{X}(-\psi)^{p+i}\omega_{\f}^{i}\wedge\omega_{\psi}^{n-i}\right)^{\frac{1}{p+i}}\right)\left(\int_{X}(-\f)^{\sigma_{n}}\omega_{\f}^{n}\right)^{\frac{p}{p+n}}
$$ and since $0\leq(-\f)\leq1$ and $\sigma_{n}>p$ we have
$$
\leq c_{p,n}\left(\prod_{i=1}^{n}\left(\int_{X}(-\psi)^{p+i}\omega^{i}\wedge\omega_{\psi}^{n-i}\right)^{\frac{1}{p+i}}\right)\left(\int_{X}(-\f)^{p}\omega_{\f}^{n}\right)^{\frac{p}{p+n}}
$$ Each term in the product is bounded since $\psi\in Ch^{p} (X,\omega)$ so we have
$$
\int_{X}(-\f)^{p+n}\omega_{\psi}^{n}\leq A \left(\int_{X}(-\f)^{p}\omega_{\f}^{n}\right)^{\frac{p}{p+n}}
$$
and the conclusion follows by applying this inequality to the extremal function $\f=h^{*}_{\omega,E}$,
where $E \subset X$ is an arbitrary compact set.
\end{proof}

\begin{rem}
 In the case where $Ch^p (X,\omega)\subset L^{q}(\mu)$ and $q\geq p+n-1$, $p>1$ as an immediate consequence of Corollary \ref{cor:comparaison} and Proposition \ref{prop:measurecapacity} we obtain that there exists $A>0$ such that $\mu\leq ACap_{\omega}^{\alpha}$ where $\alpha=(1-\frac{1}{p})^{n}$.
\end{rem}

\subsection{Examples}

It follows from Corollary \ref{cor:comparaison} that the classes $\mathcal Ch^p (X,\omega)$
and $\mathcal E^p (X,\omega)$ coincide when $n=1$.
We describe in this section the finite Choquet energy classes in special cases.

\subsubsection{Compact singularities}

The class $\mathcal Ch^p (X,\omega)$ is similar to $\mathcal E^{p} (X,\omega)$
for functions with "compact singularities":

\begin{prop}  \label{prop:compact}
let $D$ be an ample $\Q$-divisor.
Let $\f$ be a $\omega$-psh function which is bounded in a neighborhood of $D$. Then
$$
\f \in \mathcal Ch^p (X,\omega) \Longleftrightarrow \f \in \mathcal E^{p} (X,\omega).
$$
\end{prop}

The inclusions
$\mathcal E^{p+n-1} (X,\omega) \subset \mathcal Ch^p (X,\omega) \subset \mathcal E^p (X,\omega)$
are strict  in general when $n \geq 2$, as we show in Example \ref{exa:contre} below .

\begin{proof}
Let $V$ be a neighborhood of $D$ where $\f$ is bounded.
For simplicity we assume that $c_1(D)=\{\omega\}$. Let $\omega'$ be a smooth semi-positive closed
form cohomologous to $\omega$, such that $\omega' \equiv 0$ outside $V$. Let $\rho$ be a smooth
$\omega$-psh function such that $\omega'=\omega+dd^c \rho$. Shifting by a constant, we can
assume that $0 \leq \rho \leq M$. Observe that
{\small
\begin{eqnarray*}
-dd^c (-\f)^{p+j} &=&-(p+j)(p+j-1) (-\f)^{p+j-2} d \f \wedge d^c \f +(p+j) (-\f)^{p+j-1} \omega_\f \\
&\leq&  (p+j) (-\f)^{p+j-1} \omega_\f.
\end{eqnarray*}
}
Therefore
\begin{eqnarray*}
\int (-\f)^{p+j} \omega_{\f}^{n-j} \wedge \omega^j &=&
\int (-\f)^{p+j} \omega_{\f}^{n-j} \wedge \omega^{j-1} \wedge \omega'  \\
& + & \int -(-\f)^{p+j} \omega_{\f}^{n-j} \wedge \omega^{j-1} \wedge dd^c \rho \\
&=& O(1)+\int \rho \, dd^c [-(-\f)^{p+j}] \wedge \omega_{\f}^{n-j} \wedge \omega^{j-1}  \\
& \leq & O(1)+(p+j)M \int (-\f)^{p+j-1} \omega_{\f}^{n-j+1} \wedge \omega^{j-1}.
\end{eqnarray*}

We denote here by $O(1)$ the first term
$\int (-\f)^{p+j} \omega_{\f}^{n-j} \wedge \omega^{j-1} \wedge \omega' $ which is bounded, since
$\f$ is bounded on the support of $\omega'$.

By induction we obtain that each term
$\int (-\f)^{p+j} \omega_{\f}^{n-j} \wedge \omega^j$ is controlled by
$\int (-\f)^{p} \omega_{\f}^n$. Thus $\rm{Ch}_p(\f)$ is finite if and only if so is
$\int (-\f)^{p} \omega_{\f}^n$.
\end{proof}

This proposition allows us to cook up examples of $\omega$-psh functions $\f$
such that $\f \in \mathcal Ch^p(X,\omega)$ but $\f \notin \E^{p+n-1}(X,\omega)$.
The next example shows how to cook up examples such that $\f \in \E^p(X,\omega)$ but
$\f \notin \mathcal Ch^p(X,\omega)$:

\begin{exa}   \label{exa:contre}
Assume $X = \mathbb \C\P^{n - 1} \times \mathbb \C\P^1$ and $\om (x,y) := \a (x) + \b (y),$ where
$\a$ is the Fubini-Study form on $\mathbb  \C \P^{n - 1}$ and $\b$ is the Fubini-Study form on
$\mathbb \C\P^1$.
Fix $u \in PSH (\mathbb  \C\P^{n - 1},\a) \cap \mathcal C^{\infty} (\mathbb  \C\P^{n - 1})$
and  $v \in  \mathcal E (\mathbb \C\P^1,\b)$.

The function $\f$ defined by  $\f (x,y) := u (x) + v (y)$ for $(x,y) \in X$ belongs to
$\mathcal E (X,\om)$.
Moreover $\om_{\f} = \a_u + \b_v$ and for any $1 \leq \ell \leq n$, we have
$$
\om_{\f}^{n-j} = \a_u^{n-j} +(n-j) \a_u^{n-j - 1} \wedge \b_v
$$
and
$$
\om_{\f}^{n-j} \wedge \omega^j=\a_u^{n-j} \wedge \a^j+j \a^{j-1} \wedge \a_u^{n-j } \wedge \b
+(n-j) \a^j \wedge \a_u^{n-j - 1} \wedge \b_v.
$$
Thus for $j \leq n-1$,
$$
\f \in L^{p+j} (\om_{\f}^{n-j} \wedge \omega^j)
\Longleftrightarrow v \in L^{p+j} (\b_{v})
$$
hence
$$
\f \in \mathcal Ch^p (X,\omega) \Longleftrightarrow v \in \E^{p+n-1}(\C\P^1,\b)
$$
while
$$
\f \in \E^p (X,\omega) \Longleftrightarrow v \in \E^{p}(\C\P^1,\b).
$$

Choosing  $v \in L^p (\b_v) \setminus L^{p+n-1} (\b_{v})$,
we obtain an example of a $\om$-psh function $\f$
such that $\f \in \mathcal E^p (X,\om)$ but $\f \notin   \mathcal Ch^p(X,\om).$
\end{exa}

\begin{rem}
In the above examples, we can choose $u$ and $v$ toric, hence   both inclusions in Corollary \ref{cor:comparaison}
are sharp in the toric setting as well.
For details on toric singularities, we refer the reader to \cite{G14,DN15}.
\end{rem}

\subsubsection{Divisorial singularities}

Let  $D$ be an ample $\Q$-divisor, $s$ a holomorphic defining section of $L_D$
 and $h$ a smooth positive metric of $L$. We assume for simplicity that the curvature of
$h$ is $\omega$, so that the Poincar\'e-Lelong formula can be written
$$
dd^c \log |s|_h=[D]-\omega,
$$
where $[D]$ denotes the current of integration along $D$.

Let $\chi$ be a smooth convex increasing function and set $\f=\chi \circ \log |s|_h$.
We normalize $h$ so that $\chi' \circ \log |s|_h \leq 1/2$. It follows that $\f$ is strictly $\omega$-psh, since
$$
dd^c \f=\chi'' \circ L \, dL \wedge d^c L+\chi' \circ L \, dd^c L \geq -\chi' \circ L \, \omega
\geq -\omega/2,
$$
where $L:=\log |s|_h$.

\begin{prop} \label{prop:div}
Set $\f=\chi \circ \log |s|_h \in PSH(X,\omega)$. Then
$$
\f  \in \mathcal Ch^p (X,\omega)
\Longleftrightarrow
\f \in \E^{p+n-1}(X,\omega).
$$
\end{prop}

\begin{proof}
 Set $L=\log |s|_h$. Observe that
$$
\omega+dd^c \f=\chi'' \circ L \, dL \wedge d^c L+\chi' \circ L \, [D]+(1-\chi' \circ L) \omega.
$$

A necessary condition for $\f$ to belong to a finite energy class is that $\omega_\f$ does not charge pluripolar sets, hence $\chi'(-\infty)=0$ and
$$
\omega+dd^c \f=\chi'' \circ L \, dL \wedge d^c L+(1-\chi' \circ L) \omega.
$$
Since $\frac{1}{2} \leq 1-\chi' \circ L \leq 1$, we infer
$$
\omega_{\f}^{n-j} \wedge \omega^j \sim \chi'' \circ L \, dL \wedge d^c L \wedge \omega^{n-1}
+\omega^n,
$$
for $0 \leq j \leq n-1$. We write here $\mu \sim \mu'$ if the  positive Radon measures $\mu,\mu'$
are uniformly comparable, i.e. $C^{-1} \mu \leq \mu' \leq C \mu$ for some constant $C>0$.
Thus
\begin{eqnarray*}
\f \in  \mathcal Ch^p (X,\omega)
&\Longleftrightarrow&
\f \in L^{p+n-1}(\chi'' \circ L \, dL \wedge d^c L \wedge \omega^{n-1}) \\
&\Longleftrightarrow&
\f \in L^{p+n-1}(\omega_\f^n)
\Longleftrightarrow
\f \in \E^{p+n-1}(X,\omega).
\end{eqnarray*}
\end{proof}

\begin{exa}
 For $\chi(t)=-(-t)^\a, \; 0 <\a<1$, we obtain
$$
\f=-(-\log |s|_h)^\a \in \E^p(X,\omega)
\text{ iff } \a< \frac{1}{p+1}
$$
and
$$
\f=-(-\log |s|_h)^\a \in \mathcal Ch^p(X,\omega)
\text{ iff } \a< \frac{1}{p+n}
$$
\end{exa}

We refer the reader to \cite{DN15} for more information on Monge-Amp\`ere measures
with divisorial singularities.


\begin{thebibliography}{widestlabel}

\addcontentsline{toc}{chapter}{Bibliography}


\bibitem [BT82] {BT82} E.~Bedford, B.~A.~Taylor: A new capacity for plurisubharmonic functions. Acta Math. {\bf 149} (1982), no. 1-2, 1--40.


\bibitem [BGZ08] {BGZ08} S.~Benelkourchi, V.~Guedj, A.~Zeriahi: A priori estimates for weak solutions of complex Monge-Amp{\`e}re equations.  Ann. Scuola Norm. Sup. Pisa C1. Sci. ({\bf 5}), Vol VII (2008), 1-16.

\bibitem [BGZ09] {BGZ09} S.~Benelkourchi, V.~Guedj,  A.Zeriahi:
Plurisubharmonic functions with weak singularities.
Complex analysis and digital geometry, 57-74, Acta Univ. Upsaliensis Skr. Uppsala Univ. C Organ. Hist., 86,
Proceedings of the conference in honor of C.Kiselman (Kiselmanfest, Uppsala, May 2006) Uppsala Universitet, Uppsala (2009).

\bibitem [BBGZ13] {BBGZ13} R.~Berman, S.~Boucksom, V.~Guedj, A.~Zeriahi: A variational approach to complex Monge-Amp\`ere equations. Publ.Math.I.H.E.S. {\bf 117} (2013), 179-245.

\bibitem [BOUR]{BOUR} N.Bourbaki: El\'ements de math\'ematiques, Topologie g\'en\'erale,
  Fsc VIII, livre III Chap 9.

\bibitem [BEG13]{BEG13} S.~Boucksom, P.~Eyssidieux, V.~Guedj:
An introduction to the K\"ahler-Ricci flow.
Lecture Notes in Math., {\bf 2086} , Springer, Heidelberg (2013).



\bibitem [DN15] {DN15} E.DiNezza: Finite pluricomplex energy measures. Preprint arXiv:1501.03747 (2015).

\bibitem [EGZ08] {EGZ08} P.~Eyssidieux, V.~Guedj, A.~Zeriahi:  A
priori $L^{\infty}$-estimates for degenerate complex Monge-Amp{\`e}re
equations. International Mathematical Research Notes, Vol. {\bf 2008}, Article ID rnn070, 8 pages.

\bibitem [EGZ09] {EGZ09} P.~Eyssidieux, V.~Guedj, A.~Zeriahi: Singular K\"ahler-Einstein metrics. J. Amer. Math. Soc. {\bf 22} (2009), 607-639.
\bibitem [G14]  {G14} V.Guedj: The metric completion of the Riemannian space of K\"{a}hler metrics. Preprint arXiv:1401.7857.v2 (2014).

\bibitem [GZ05] {GZ05} V.~Guedj, A.~Zeriahi: Intrinsic capacities on compact K{\"a}hler manifolds. J. Geom. Anal.  {\bf 15}  (2005),  no. 4, 607-639.

\bibitem [GZ07] {GZ07} V.~Guedj, A.~Zeriahi: The weighted Monge-Amp{\`e}re energy
   of quasiplurisubharmonic functions. J. Funct. An.  {\bf 250} (2007), 442-482.




\end{thebibliography}
\end{document}